\documentclass{article}
\usepackage{graphicx} 
\usepackage{lipsum}
\usepackage{amsfonts,amsmath,amssymb,amsthm,color,enumerate,setspace,tikz}
\usepackage[unicode=true]
 {hyperref}
\hypersetup{
colorlinks=true,
urlcolor=black,
citecolor=blue,
linkcolor=blue,
}

\newtheorem{theorem}{Theorem}[section]
\newtheorem{proposition}[theorem]{Proposition}
\newtheorem{lemma}[theorem]{Lemma}

\newtheorem{question}[theorem]{Question}

\theoremstyle{definition}
\newtheorem{remark}[theorem]{Remark}

\newcommand{\Hol}{\mathrm{Hol}}
\newcommand{\Aut}{\mathrm{Aut}}

\providecommand{\keywords}[1]{\textbf{Keywords:} #1}

\numberwithin{equation}{section}
\addtocontents{toc}{\setcounter{tocdepth}{1}} 

\allowdisplaybreaks

\title{Hopf--Galois structures of cyclic type on \\parallel extensions of prime power degree}
\author{Andrew Darlington
\and
Cindy (Sin Yi) Tsang}

\newcommand{\Addresses}{{
\bigskip
\noindent A. Darlington, \textsc{Department of Mathematics and Data Science, Vrije Universiteit Brussel, Pleinlaan 2, 1050, Brussels, Belgium}\par\nopagebreak
\noindent\textit{Email} \texttt{andrew.darlington@vub.be}\\
\noindent\textit{Homepage} \url{https://sites.google.com/view/andrewdarlington/}

\medskip

\noindent C. Tsang, \textsc{Department of Mathematics, Ochanomizu University, 2-1-1 Otsuka, Bunkyo-ku, Tokyo,
Japan}\par\nopagebreak
\noindent\textit{Email}  \texttt{tsang.sin.yi@ocha.ac.jp}\\
\noindent\textit{Homepage} \url{https://sites.google.com/site/cindysinyitsang/}
}}

\begin{document}

\date{}
\maketitle

\begin{abstract} Let $L/K$ be any finite separable extension with normal closure $\widetilde{L}/K$. An extension $L'/K$ is said to be \textit{parallel to $L/K$} if $L'$ is an intermediate field of $\widetilde{L}/K$ with $[L':K]=[L:K]$. We study the following question --- Given that $L/K$ admits a Hopf--Galois structure of type $N$, does it imply that every extension parallel to $L/K$ also admits a Hopf--Galois structure of type $N$? We completely solve this problem when the degree $[L:K]$ is a prime power and the type $N$ is cyclic. Our approach is group-theoretic and uses the work of Greither--Pareigis and Byott.
\end{abstract}

\medskip

\keywords{Hopf--Galois structures, cyclic type, parallel extensions,\\ \mbox{\hspace{2.45cm}} holomorph, regular subgroups, transitive subgroups}

\medskip


\section{Introduction}

Hopf--Galois structures were first described by Chase and Sweedler in \cite{CS}. The original motivation was to study purely inseparable extensions, but it was soon realised that this approach was not fruitful. Nevertheless, the theory also applies to separable extensions, in which case the Hopf--Galois structures admit a group-theoretic classification, thanks to the work of Greither and Pareigis \cite{GP}. Below, let us explain this in more detail (also see \cite{book1}).

\medskip

Let $L/K$ be a finite separable extension with normal closure $\widetilde{L}/K$. We have the Galois groups $G=\mathrm{Gal}(\widetilde{L}/K)$ and $G'=\mathrm{Gal}(\widetilde{L}/L)$. The result of \cite{GP} states that the Hopf--Galois structures on $L/K$ (up to isomorphism) are in one-to-one correspondence with the \textit{regular} subgroups $N$ of $\mathrm{Sym}(G/G')$, i.e. the transitive subgroups with trivial stabilisers, that are normalised by the subgroup $\lambda(G)$ of left translations, where
\[ \lambda : G \rightarrow \mathrm{Sym}(G/G');\quad \lambda(g) = (hG' \mapsto ghG').\]
More specifically, the Hopf--Galois structure on $L/K$ associated to $N$ is defined to be the sub-Hopf algebra $(\widetilde{L}[N])^G$ of $\widetilde{L}[N]$ over $K$ consisting of the elements that are fixed by the action of $G$, where $G$ acts on $\widetilde{L}$ via the Galois group and on $N$ via conjugation by $\lambda(G)$. The action of $(\widetilde{L}[N])^G$ on $L$ is given by
\[ \left(\sum_{\sigma\in N} \ell_\sigma\sigma \right)\cdot x = \sum_{\sigma\in N}\ell_\sigma g_\sigma(x)\quad \left(\forall \sigma \in N : \sigma(g_\sigma G') = G'\right)\]
for all $x\in L$. The group $N$ or its isomorphism class is referred to as the \textit{type} of the associated Hopf--Galois structure. Note that 
\[|N| = [G:G']=[L:K]\]
holds. The symmetric group $\mathrm{Sym}(G/G')$ is large and could be difficult to work with. By fixing the type $N$ in advance and by reversing the roles of $G$ and $N$, Byott \cite{Byott} reformulated this correspondence in terms of the holomorph 
\[ \Hol(N) = N\rtimes \Aut(N)\]
of $N$, which is much smaller than $\mathrm{Sym}(G/G')$. One consequence of his result is that the following statements are equivalent:
\begin{enumerate}[(1)]
\item The extension $L/K$ admits a Hopf--Galois structure of type $N$.
\item The group $G$ is isomorphic to a transitive subgroup $T$ of $\Hol(N)$ under an isomorphism that takes $G'$ to the stabiliser $\mathrm{Stab}_T(1_N)$.
\end{enumerate}
This is the point of view that we shall take in this paper.

\medskip

With the same set-up as above, an extension $L'/K$ is said to be \textit{parallel to $L/K$} if $L'$ is an intermediate field of $\widetilde{L}/K$ with $[L':K]=[L:K]$. The notion of ``parallel" is not symmetric because $L$ need not be contained in the normal closure of $L'/K$. Also clearly $L/K$ has no parallel extension except itself when it is normal. In \cite{Darlington}, the first-named author initiated the study of comparing the Hopf--Galois structures on $L/K$ and those on a parallel extension $L'/K$. More precisely, in \cite[Section 4]{Darlington}, he considered the following problem:

\begin{question}\label{Q:HGS0} If $L/K$ admits a Hopf--Galois structure, does it imply that its parallel extensions $L'/K$ all admit a Hopf--Galois structure?
\end{question}

Although counterexamples exist, computation by \textsc{Magma} \cite{magma} suggests that the answer to Question \ref{Q:HGS0} is often affirmative (see \cite[Table 5]{Darlington}), and is always affirmative when $[L:K]$ is squarefree (see \cite[Conjecture 4.2]{Darlington}). The squarefree degree case is somewhat tractable because there is a classification of groups of squarefree order by \cite{MurtyMurty}, but the general case can be extremely difficult.

\medskip

In this paper, we shall refine Question \ref{Q:HGS0} by fixing the type $N$ in advance. We ask the following question, which seems much more approachable.

\begin{question}\label{Q:HGS} If $L/K$ admits a Hopf--Galois structure of type $N$, does it imply that its parallel extensions $L'/K$ all admit a Hopf--Galois structure of type $N$?
\end{question}

Following \cite{Darlington}, we approach Question \ref{Q:HGS} group-theoretically, as follows. The hypothesis that $L/K$ admits a Hopf--Galois structure of type $N$ means that we may identify $G$ as a transitive subgroup of $\Hol(N)$ and $G' = \mathrm{Stab}_G(1_N)$. Now, the extensions parallel to $L/K$ are exactly the fixed fields of the subgroups $H$ of $G$ of index $[L:K]$. For each subgroup $H$ of $G$, the normal closure of $L^H/K$ is the fixed field of the core $C=\mathrm{Core}_G(H)$ of $H$ in $G$, i.e. the largest normal subgroup of $G$ contained in $H$. Let us summarise the set-up in a diagram:

{\small
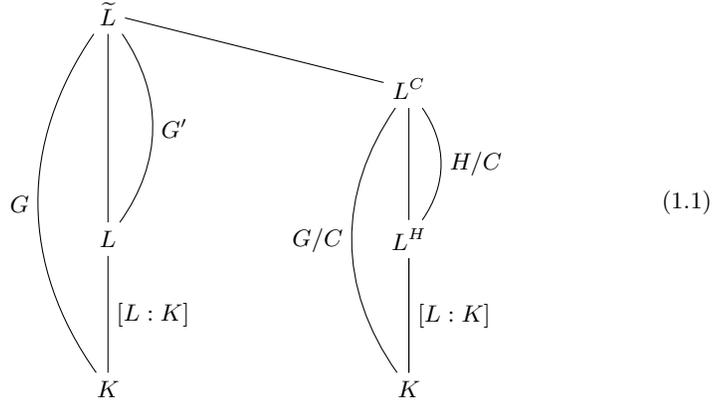
\begin{equation}\label{diagram} \begin{tikzpicture}[baseline={([yshift=-.8ex]current bounding box.center)},scale=1]
    \node at (0,0) [name=K] {$K$};
    \node at (0,2) [name=L] {$L$};
    \node at (0,5) [name=E] {$\widetilde{L}$};
    \node at (4,2) [name=H] {$L^H$};
    \node at (4,4) [name=C] {$L^C$};
    \node at (4,0) [name=K'] {$K$};
    \draw[-] (L) -- (E) (K') -- (H) -- (C) -- (E); 
    \draw[-] (K) to node[right] {$[L:K]$} (L);
    \draw (K) to[bend left=35] node[left] {$G$} (E);
    \draw (L) to[bend right=35] node[right] {$G'$} (E);
    \draw (K') to[bend left=35] node[left] {$G/C$} (C);
    \draw (H) to[bend right=35] node[right] {$H/C$} (C);
    \draw[-] (K') to node[right] {$[L:K]$} (H);
\end{tikzpicture}\end{equation}\vspace{2mm}}

\noindent We then see that $L^H/K$ admits a Hopf--Galois structure of type $N$ if and only if $G/C$ is isomorphic to a transitive subgroup of $\Hol(N)$ under an isomorphism that takes $H/C$ to the stabiliser of $1_N$. 
It follows that Question \ref{Q:HGS} reduces to:

\begin{question}\label{Q:HSG'} Let $G$ be a transitive subgroup of $\Hol(N)$ with $G'=\mathrm{Stab}_G(1_N)$. For any  subgroup $H$ of $G$ of index $|N|$ with $C=\mathrm{Core}_G(H)$, is $G/C$ isomorphic to a transitive subgroup $T$ of $\Hol(N)$ under an isomorphism that maps $H/C$ to the stabiliser $\mathrm{Stab}_T(1_N)$?
\end{question}

In the case that $H = g^{-1}G'g$ is conjugate to $G'$ with $g\in G$, or equivalently $L^H$ is conjugate to $L$ in the set-up \eqref{diagram}, we have $C=1$ and conjugation by $g$ is an isomorphism from $G$ to itself that sends $H$ to $G'$. Hence, if $L/K$ admits a Hopf--Galois structure of type $N$, then so do the extensions that are conjugate to $L/K$. The same holds for the $H$ that lie in the same $\Aut(G)$-orbit as $G'$.

\medskip

The purpose of this paper is to study Hopf--Galois structures of cyclic type on parallel extensions of prime power degree. As in many situations, the cases of odd and even prime powers behave very differently. Our main results are:

\begin{theorem}\label{thm:odd}
    Let $L/K$ be any finite separable extension of odd prime power degree admitting a Hopf--Galois structure of cyclic type. For any extension $L'/K$ parallel to $L/K$, the following are equivalent:
    \begin{enumerate}[$(1)$]
    \item $L'/K$ admits a Hopf--Galois structure of cyclic type.
    \item $L'/K$ is conjugate to $L/K$.
    \end{enumerate}
\end{theorem}
\begin{proof}
   This follows from Proposition \ref{prop:odd}.
\end{proof}

\begin{theorem}\label{thm:even} Let $L/K$ be any finite separable extension of even prime power degree admitting a Hopf--Galois structure of cyclic type. Let $G$ denote the Galois group of the normal closure of $L/K$. Then $|G|=2^s[L:K]$ is also a power of $2$, and the following hold:
\begin{enumerate}[$(1)$]
\item If $s=1$ and $G$ has an element of order $[L:K]$, then every extension $L'/K$ parallel to $L/K$ admits a Hopf--Galois structure of cyclic type.
\item If $s=1$ and $G$ has no element of order $[L:K]$, or if $s\geq 2$, then there is a normal extension $L'/K$ parallel to $L/K$ that does not admit a Hopf--Galois structure of cyclic type.
\end{enumerate}
The case $s=0$ is irrelevant because then $L/K$ is normal.
\end{theorem}
\begin{proof}
    This follows from Proposition \ref{prop:p=2}.
\end{proof}

In the setting of Theorem \ref{thm:even}, we can in fact give a complete characterisation, which is group-theoretic, of the (not necessarily normal) extensions $L'/K$ parallel to $L/K$ that do not admit any Hopf--Galois structure of cyclic type. To that end, we use the set-up \eqref{diagram}, where we identify $G$ as a transitive subgroup of $\Hol(N)$ and $G'=\mathrm{Stab}_G(1_N)$ for $N$ cyclic of order $[L:K]$.

\begin{theorem}\label{thm:H} Let $N$ be the cyclic group of order $2^e$ and let $G$ be a transitive subgroup of $\Hol(N)$. For any subgroup $H$ of $G$ of index $2^e$ with $C = \mathrm{Core}_G(H)$, the following are equivalent:
\begin{enumerate}[$(1)$]
\item $G/C$ is not isomorphic to any transitive subgroup $T$ of $\Hol(N)$ under an isomorphism that takes $H/C$ to the stabiliser $\mathrm{Stab}_T(1_N)$.
\item Any one of the following holds:
\begin{enumerate}[$(i)$]
\item $|H\cap N|\geq 4$;
\item $|H\cap N| = 2$ and $G$ has no element of order $2^e$;
\item $|H\cap N|=2$ and $H$ is not normal in $G$;
\item $|H\cap N|=1$ with $H = \langle[\sigma^u,\varphi_{-1}]\rangle $ for an odd integer $u$, and 
\begin{align*} 
&|G| = 2^{e+1},\,\  \mathrm{Stab}_G(1_N) = \langle\varphi_{1+2^{e-1}}\rangle,\mbox{ and}\\
&\hspace{1cm}\mbox{either }|G\cap N|\geq 8\mbox{ or }|G\cap N| = |[G,G]|=4,
\end{align*}
where $\sigma$ is a generator of $N$, and $\varphi_a$ denotes the automorphism on $N$ defined by $\varphi_a(\sigma)=\sigma^a$ for each odd integer $a$.
\end{enumerate}
\end{enumerate}
\end{theorem}
\begin{proof}
    This follows from Propositions \ref{prop:4}, \ref{prop:2}, and \ref{prop:1}.
\end{proof}

\begin{remark} Let $N$ be a cyclic group. It was shown in \cite[Theorem 3.9]{Darlington} that Question \ref{Q:HGS} admits a positive answer when $|N|$ is the product of two distinct primes. Our results show that the behaviour is on the other extreme when $|N|$ is a  prime power, especially odd prime power. Also, let us remark that for $|N|$ squarefree, the answer to Question \ref{Q:HGS} is ``no" in general when there are three or more prime factors, by calculations in \textsc{Magma} \cite{magma}.
\end{remark}

\begin{remark} The proof of \cite[Lemma 3.1]{Darlington}, which is part of \cite[Theorem 3.9]{Darlington}, has a small gap. It cites \cite[Theorem 1]{Byott}, which only holds for normal separable extensions. Nevertheless, the statement is still true by the following simple fact --- when $N$ is cyclic of order $pq$, where $p>q$ are primes with $p\not\equiv 1\pmod{q}$, for any subgroup $G$ of $\Hol(N)$, the subgroups of $G$ of index $pq$ are Hall subgroups and so are conjugates of each other because $G$ is soluble here. We highlight that the lemma also carries the implicit hypothesis that the original extension admits a Hopf--Galois structure in the first place.
\end{remark}

\section{Subgroups of the holomorph}

In this section, let $N$ be a finite group. We shall assume that $\Hol(N)$ contains a unique Hall $\pi$-subgroup $Q$, where $\pi$ is the set of prime divisors of $|N|$. By the Schur--Zassenhaus theorem, we know that $\Hol(N) =Q\rtimes X$ for some subgroup $X$ of order coprime to $|N|$. For example, this is the case when $N$ is cyclic and when $N$ has squarefree order (see \cite[Lemma 2.4]{Darlington}).

\medskip

Under the above hypothesis, we can restrict to transitive subgroups of $Q$ in some situations. The next two lemmas are needed for the proof of Theorem \ref{thm:odd}. But they are irrelevant for Theorems \ref{thm:even} and \ref{thm:H} because $\Aut(N)$ is a $2$-group when $N$ is cyclic of order a power of $2$.

\begin{lemma}\label{lem:Q part}
    Let $G$ be a subgroup of $\Hol(N)$ and let $H$ be any subgroup of $G$.
    \begin{enumerate}[$(a)$]
    \item If $G$ is transitive, then $G\cap Q$ is also transitive.
    \item If the prime factors of $[G:H]$ divide $|N|$, then $[G:H] = [G\cap Q: H\cap Q]$.
    \end{enumerate}
    \end{lemma}
\begin{proof} To prove (a), observe that
\begin{align*}
&[G:G\cap Q][G\cap Q:\mathrm{Stab}_{G\cap Q}(1_N)] \\
&\hspace{1cm} = [G:\mathrm{Stab}_{G\cap Q}(1_N)]\\
&\hspace{1cm} =[G:\mathrm{Stab}_G(1_N)][\mathrm{Stab}_G(1_N):\mathrm{Stab}_{G\cap Q}(1_N)].
\end{align*}
Note that $[G:G\cap Q]$ is coprime to $|N|$ because $G/G\cap Q$ embeds into $X$. If $G$ is transitive, then $[G:\mathrm{Stab}_G(1_N)] = |N|$, and we deduce that
\[ [G\cap Q : \mathrm{Stab}_{G\cap Q}(1_N)] = |N|\]
must also hold, namely $G\cap Q$ is transitive. We remark that the argument here is due to \cite[Lemma 2.1]{Crespo}.

\medskip

To prove (b), observe that
\[ [GQ:HQ][G\cap Q:H\cap Q] = \frac{|G||Q|}{|H||Q|}= [G:H].\]
Note that $[GQ:HQ]$ is coprime to $|N|$ because $GQ/Q$ embeds into $X$. If the prime factors of $[G:H]$ divide $|N|$, we must then have
\[[GQ:HQ]=1,\quad [G:H]=[G\cap Q:H\cap Q],\]
which is as claimed.
\end{proof}

The next lemma is basically \cite[Proposition 2.6(i)]{Darlington}; although $|N|$ is assumed to be squarefree in \cite[Section 2]{Darlington}, most of the arguments there are still valid as long as $\Hol(N)$ has a unique Hall $\pi$-subgroup. We include a proof here because we are loosening some of the hypotheses of \cite{Darlington}.

\medskip

We first make an observation. Let $G$ be a subgroup of $\Hol(N)$. Notice that $G\cap Q$ is a normal Hall $\pi$-subgroup of $G$ because $G/G\cap Q$ embeds into $X$. By the Schur--Zassenhaus theorem, we can then write $G = (G\cap Q)\rtimes Y$, where $Y$ is a subgroup of order coprime to $|N|$. Similarly, for any subgroup $H$ of $G$ and the normaliser $N_G(H\cap Q)$ of $H\cap Q$ in $G$, we may write
\begin{align*}
H  &= (H\cap Q)\rtimes V,\\
N_G(H\cap Q)& = (N_G(H\cap Q)\cap Q)\rtimes W, 
\end{align*}
where $V$ and $W$ have orders coprime to $|N|$. In the case that the prime factors of $[G:H]$ divide $|N|$, we must have $|Y| = |V|$ by Lemma \ref{lem:Q part}. Since $N_G(H\cap Q)$ contains $H$, the prime factors of $[G:N_G(H\cap Q)]$ also divide $|N|$, so again we have $|Y| = |W|$ by Lemma \ref{lem:Q part}. We then see that
\[ N_G(H\cap Q) = (N_G(H\cap Q)\cap Q) \rtimes V\]
holds by order consideration. In other words, we can take $W = V$.

\begin{lemma}\label{lem:Q part'} Let $G$ be a subgroup of $\Hol(N)$ and let $H_1,\, H_2$ be any subgroups of $G$ such that the prime factors of their indices $[G:H_1],\, [G:H_2]$ divide $|N|$. The following are equivalent:
\begin{enumerate}[$(1)$]
\item $H_1$ and $H_2$ are conjugate in $G$.
\item $H_1\cap Q$ and $H_2\cap Q$ are conjugate in $G$.
\end{enumerate}
\end{lemma}
\begin{proof} If $H_1$ and $H_2$ are conjugate in $G$, then clearly  $H_1\cap Q$ and $H_2\cap Q$ are also conjugate in $G$ because $Q$ is normal in $\Hol(N)$. If $H_1\cap Q$ and $H_2\cap Q$ are conjugate in $G$, then we apply the above observation and write
\begin{align*}
    H_1& = (H_1\cap Q) \rtimes V_1,\\
    H_2&=(H_2\cap Q)\rtimes V_2,\quad \ N_G(H_2\cap Q) = (N_G(H_2\cap Q)\cap Q)\rtimes V_2,
\end{align*}
where $|V_1|=[G:G\cap Q]=|V_2|$ is coprime to $|N|$. The hypothesis here is that  $H_2\cap Q = g(H_1\cap Q)g^{-1}$ for some $g\in G$. But then 
\[ gV_1g^{-1} \subseteq gN_G(H_1\cap Q)g^{-1} =  N_G(H_2\cap Q),\]and so by order consideration, we have
\[ N_G(H_2\cap Q) = (N_G(H_2\cap Q)\cap Q) \rtimes (gV_1g^{-1}).\]
We then deduce from the Schur--Zassenhaus theorem that $V_2 = hgV_1g^{-1}h^{-1}$ for some $h\in N_G(H_2\cap Q)$. As a consequence, we have
\[(H_2\cap Q)\rtimes V_2 = hg( (H_1\cap Q)\rtimes V_1) g^{-1}h^{-1},\]
whence $H_1$ and $H_2$ are conjugate in $G$. 
\end{proof}

\begin{remark}
In the proof of Lemma \ref{lem:Q part'} given in \cite[Proposition 2.6(i)]{Darlington}, the desired $h\in N_G(H_2\cap Q)$ was derived from the fact that $gV_1g^{-1}$ and $V_2$ are both Hall $\pi'$-subgroups of $N_G(H_2\cap Q)$. In \cite[Section 2]{Darlington}, since $|N|$ is assumed to be squarefree, indeed $\Hol(N)$ is soluble and the Hall $\pi'$-subgroups of $N_G(H_2\cap Q)$ are conjugates. Our proof shows that it suffices to apply the Schur--Zassenhaus theorem and solubility of $\Hol(N)$ is not required.
\end{remark}

\section{Notation and preliminaries}

In the rest of this paper, let $N = \langle \sigma \rangle$ be a cyclic group of prime power order $p^e$ with $e\geq 2$. The case $e=1$ can be disregarded --- it is trivial for $p$ odd because the index $p$ subgroups of any $G\leq \Hol(N)$ are conjugates of each other by the Schur--Zassenhaus theorem, and is irrelevant for $p=2$ because $\Hol(N) = N$.

\medskip

For any integer $a$ coprime to $p$, let us define
\[ \varphi_a : N \rightarrow N;\quad \varphi_a(\sigma) = \sigma^a,\]
which lies in $\Aut(N)$. It shall be helpful to recall that:
\begin{enumerate}[$\bullet$]
\item If $p$ is odd, then $\Aut(N) \simeq C_{p^{e-1}(p-1)}$, and its Sylow $p$-subgroup is the subgroup consisting of the $\varphi_a$ for which $a\equiv 1\pmod{p}$.
\item If $p=2$, then $\Aut(N)\simeq C_2\times C_{2^{e-2}}$, or to be precise $\Aut(N) = \langle\varphi_{-1}\rangle \times \langle \varphi_5\rangle$, where $\langle \varphi_5\rangle$ is the subgroup consisting of the $\varphi_a$ for which $a\equiv 1\pmod{4}$.
\end{enumerate}
We shall write elements of $\Hol(N)$ in the form $[\sigma^u,\varphi_a]$, where $u$ and $a$ are any integers with $a$ coprime to $p$. Then the multiplication in $\Hol(N)$ is given by
\[ [\sigma^u,\varphi_a][\sigma^v,\varphi_b]= [\sigma^{u+va},\varphi_{ab}].\]
For any non-negative integer $k$, let us further define
\[ S(a,k) =\frac{a^k-1}{a-1}=1 + a + a^2 + \cdots + a^{k-1}.\]
Then powers in $\Hol(N)$ may be computed via the formula
\begin{equation}\label{eqn:power}
    [\sigma^u,\varphi_a]^{k} = [\sigma^{uS(a,k)},\varphi_{a^k}].
\end{equation}
For any integer $m$, define $v_p(m)$ to be the $p$-adic valuation of $m$, namely $p^{v_p(m)}$ is the exact power of $p$ dividing $m$, and $v_p(0) =\infty$ by convention.

\begin{lemma}\label{lem:valuation} Let $a$ be an integer with $a\equiv 1\pmod{p}$. 
\begin{enumerate}[$(a)$]
\item If $p$ is odd, then for any non-negative integer $k$, we have
\begin{equation}\label{eqn:valuation}
v_p(S(a,k)) = v_p(k).
\end{equation}
\item If $p=2$, then for any non-negative integer $k$, we have
\begin{equation}\label{eqn:valuation2}
    v_2(S(a,k)) =
    \begin{cases}
    v_2(k)&\mbox{if $a\equiv1\hspace{-2mm}\pmod{4}$ or $k$ is odd},\\
    v_2(k) + v_2(\frac{a+1}{2}) &\mbox{if $a\equiv3\hspace{-2mm}\pmod{4}$ and $k$ is even}.
    \end{cases}
\end{equation}
\end{enumerate}
\end{lemma}
\begin{proof} For $p$ odd and for $p=2 $ with $a\equiv1\pmod{4}$, one can find proofs in \cite[Lemma 4]{Rump1}, \cite[Lemma 2.1]{Tsang}, or \cite[Lemma 2.17]{CCC}, for example (\cite{CCC} only treats the odd case). For $p=2$ with $a\equiv 3\pmod{4}$, suppose first that $k$ is odd. Then
\[
S(a,k) \equiv \sum_{i=0}^{k-1}a^i \equiv \sum_{i=0}^{k-1}1\equiv k\hspace{-2mm}\pmod{2}.
\]
This means that $S(a,k)$ is also odd, namely $v_2(S(a,k)) = 0 = v_2(k)$. Suppose now that $k$
is even. Then $\frac{k}{2}$ is an integer. Since $a^2\equiv 1\pmod{4}$ and
\begin{align*}
    S(a,k) &= \frac{a^k-1}{a-1}= \frac{(a^2)^{\frac{k}{2}}-1}{a^2-1}\cdot (a+1) = S(a^2,\textstyle\frac{k}{2}) (a+1),
\end{align*}
we deduce that 
\[ v_2(S(a,k)) = v_2(S(a^2,\textstyle\frac{k}{2})) + v_2(a+1) = v_2(\textstyle\frac{k}{2}) + v_2(a+1),\]
which equals the expression in \eqref{eqn:valuation2}. 
\end{proof}

In view of \eqref{eqn:power}, the order of an element $[\sigma^u,\varphi_a]$, where $a\equiv 1\pmod{p}$, of $\Hol(N)$ may be determined using Lemma \ref{lem:valuation}, as follows. 

\begin{lemma}\label{lem:order} Let $a$ be an integer with $a\equiv 1\pmod{p}$.
\begin{enumerate}[$(a)$]
\item If $p$ is odd, then for any integer $u$, we have
\begin{equation}\label{eqn:orderp}
|[\sigma^u,\varphi_a]| = \max\{ p^{e-v_p(u)}, |\varphi_a|\}.
\end{equation}
\item If $p=2$, then for any integer $u$, we have
\begin{equation}\label{eqn:order}
|[\sigma^u,\varphi_a]| = \begin{cases}
    \max\{ 2^{e-v_2(u)}, |\varphi_a|\} & \mbox{if $a\equiv1\hspace{-2mm}\pmod{4}$},\\
    \max\{2^{e-v_2(u)-v_2(\frac{a+1}{2})},|\varphi_a|\}&\mbox{if }a\equiv3\hspace{-2mm}\pmod{4}.
\end{cases} 
\end{equation}
\end{enumerate}
\end{lemma}
\begin{proof} Note that $|\varphi_a|$ is a power of $p$ because $a\equiv 1\pmod{p}$. Thus, we deduce from \eqref{eqn:power} that the order of $[\sigma^u,\varphi_a]$ is equal to $\max\{p^f,|\varphi_a|\}$, where $f$ denotes the smallest non-negative integer for which 
\[ uS(a,p^f) \equiv 0\hspace{-2mm}\pmod{p^e},\mbox{ namely }v_p(S(a,p^f)) = e- v_p(u).\]
The claim now follows from Lemma \ref{lem:valuation}.
\end{proof}

Let $G$ be a transitive subgroup of $\Hol(N)$ and let $H$ be any subgroup of $G$ of index $p^e$ with $C=\mathrm{Core}_G(H)$. To prove our theorems, we need a method to decide whether $G/C$ is isomorphic to a transitive subgroup of $\Hol(N)$ under an isomorphism that takes $H/C$ to the stabiliser. In some cases $G/C$ is not even isomorphic to a transitive subgroup of $\Hol(N)$ because its elements have small orders. The next lemma is helpful when dealing with such cases.

\begin{lemma}\label{lem:intersection} Let $G$ be any subgroup of $\Hol(N)$ and let $H$ be any subgroup of $G$ with $C=\mathrm{Core}_G(H)$.
\begin{enumerate}[$(a)$]
\item If $|H\cap N| \geq p$, then $G/C$ has no element of order $p^{e}$.
\item If $|H\cap N| \geq 4$ and $p=2$, then $G/C$ has no element of order $2^{e-1}$.
\item If $|H\cap N| = 2$ and $p=2$, then $G/C$ has an element of order $2^{e-1}$ exactly when $G$ has an element of order $2^{e}$. 
\end{enumerate}
\end{lemma}
\begin{proof} Since the subgroups of $N$ are all characteristic, clearly $H\cap N$ lies in $C$. Let $[\sigma^u,\varphi_a]\in G$ be an element whose order in $G/C$ is a power of $p$. By raising it to a suitable power, we may assume that the order of $[\sigma^u,\varphi_a]$ in $G$ is also a power of $p$, that is $a\equiv1\pmod{p}$, without affecting its order in $G/C$. Notice that then $\varphi_a^{p^{e-1}} =1$, and $\varphi_{a}^{2^{e-2}}=1$ when $p=2$. 

\medskip

It follows immediately from \eqref{eqn:power} and Lemma \ref{lem:valuation} that
\begin{align*}
    [\sigma^u,\varphi_a]^{p^{e-1}} &=\sigma^{uS(a,p^{e-1})} \in \langle \sigma^{p^{e-1}}\rangle,\\
    [\sigma^u,\varphi_a]^{2^{e-2}} &=\sigma^{uS(a,2^{e-2})} \in \langle \sigma^{2^{e-2}}\rangle \mbox{ when }p=2,
\end{align*} 
which imply (a) and (b), respectively. Now, suppose that $p=2$. Similarly, we may deduce from  \eqref{eqn:valuation2} that 
\[[\sigma^u,\varphi_a]^{2^{e-2}}  =\sigma^{uS(a,2^{e-2})} \in \langle \sigma^{2^{e-1}}\rangle \iff a\equiv 3\hspace{-2mm}\pmod{4}\mbox{ or }u\mbox{ is even}.\]
But from \eqref{eqn:order}, we also know that
\[ |[\sigma^u,\varphi_a]| = 2^e \iff a\equiv 1\hspace{-2mm}\pmod{4}\mbox{ and $u$ is odd}.\]
The two implications above together yield (c).
\end{proof}

\section{Odd prime power case}

In this section, we assume that $p$ is an odd prime.

\begin{lemma}\label{lem:transitivep} A transitive subgroup $G$ of $\Hol(N)$ has an element of order $p^e$.
\end{lemma}
\begin{proof}
We may assume that $G$ is a $p$-group by Lemma \ref{lem:Q part}. By transitivity, we know that $G$ has an element of the form $[\sigma,\varphi_a]$, where $a\equiv 1\pmod{p}$ because $G$ is a $p$-group. Since $|\varphi_a| \leq p^{e-1}$, we see from \eqref{eqn:orderp} that $[\sigma,\varphi_a]$ has order $p^e$.
\end{proof}

\begin{proposition}\label{prop:odd} Let $G$ be a transitive subgroup of $\Hol(N)$ and let $H$ be any subgroup of $G$ of index $p^e$ with $C = \mathrm{Core}_G(H)$. The following are equivalent:
\begin{enumerate}[$(1)$]
\item $G/C$ is isomorphic to a transitive subgroup $T$ of $\Hol(N)$ under an isomorphism that sends $H/C$ to the stabiliser $\mathrm{Stab}_T(1_N)$.
\item $H$ is conjugate to $\mathrm{Stab}_G(1_N)$ in $G$.
\end{enumerate}
\end{proposition}
\begin{proof} The implication (2)$\Rightarrow$(1) is trivial. Conversely, suppose that (1) holds. By Lemmas \ref{lem:intersection}(a) and \ref{lem:transitivep}, we know that $H\cap N=1$. Then $(H\cap Q)\cap N=1$, where $Q$ is the unique Hall $p$-subgroup of $\Hol(N)$. By Lemma \ref{lem:Q part'}, it is enough to show that $H\cap Q$ is conjugate to $\mathrm{Stab}_{G\cap Q}(1_N)$ in $G$. In view of Lemma \ref{lem:Q part}, replacing $G$ and $H$ by $G\cap Q$ and $H\cap Q$, respectively, we may assume that $G$ is a $p$-group. This means that $b\equiv 1\pmod{p}$ for all $[\sigma^v,\varphi_b]\in G$, and we can also put $|G'| = |H| = p^s$, where $G' =\mathrm{Stab}_G(1_N)$. 


\medskip

The projection of $H$ onto $\Aut(N)$ is isomorphic to $H$ because $H\cap N=1$. Since $\Aut(N)$ is cyclic, we see that $H = \langle [\sigma^u,\varphi_a]\rangle$ with $a\equiv 1+p^{e-s}\pmod{p^e}$. Since $[\sigma^u,\varphi_a]$ has order $p^s$, it also follows from \eqref{eqn:power} and \eqref{eqn:valuation} that
\[ uS(a,p^{s})\equiv0\hspace{-2mm}\pmod{p^e}, \mbox{ and hence } u\equiv 0\hspace{-2mm}\pmod{p^{e-s}}.\]
We then deduce that there exists $v$ such that
\[ v(1-a) \equiv -u\hspace{-2mm}\pmod{p^e}.\]
Since $G$ is transitive, we can find $[\sigma^v,\varphi_b] \in G$, and we have
\[ [\sigma^v,\varphi_b][\sigma^u,\varphi_a][\sigma^{v},\varphi_b]^{-1} 
= [\sigma^{v(1-a) +ub},\varphi_a] = [\sigma^{u(b-1)},\varphi_a].
\]
The important thing to observe here is that
\[v_p(u(b-1)) > v_p(u)\]
because $b\equiv 1\pmod{p}$. Therefore, by repeating this process, we see that $H$ is conjugate to $\langle\varphi_a\rangle$. But $|H| = |G'|$, so necessarily $G ' =\langle\varphi_a\rangle$, and this completes the proof.
\end{proof}

\section{Even prime power case}

In this section, we assume that $p=2$. 

\begin{lemma}\label{lem:transitive2}
        A transitive subgroup $G$ of $\Hol(N)$ has an element of order $2^{e-1}$. Moreover, in the case that $G$ has no element of order $2^e$, we have 
        \begin{equation}\label{eqn:ua cong}
        b-1\equiv 2v\hspace{-2mm}\pmod{4}
        \end{equation}
        for all $[\sigma^v,\varphi_b]\in G$, and in particular $\mathrm{Stab}_G(1_N)$ is contained in $\langle\varphi_5\rangle$.
\end{lemma}
\begin{proof}Since $G$ is transitive, we can find $[\sigma,\varphi_a],[\sigma^{-1},\varphi_c]\in G$. 
\begin{enumerate}[$\bullet$]
\item If $a\equiv 1\pmod{4}$, then $[\sigma,\varphi_a]$ has order $2^e$ by \eqref{eqn:order}.
\item If $c\equiv 1\pmod{4}$, then $[\sigma^{-1},\varphi_c]$ has order $2^e$ by \eqref{eqn:order}.
\item If $a,c\equiv 3\pmod{4}$, then
\[ [\sigma^{-1},\varphi_c]^{-1}[\sigma,\varphi_a] = [\sigma^{2c^{-1}},\varphi_{c^{-1}a}]\in G\]
has order $2^{e-1}$ by \eqref{eqn:order} because $c^{-1}a\equiv 1\pmod{4}$.
\end{enumerate}
In all cases, we see that $G$ has an element of order $2^{e-1}$.

\medskip

Now, suppose that $G$ has no element of order $2^e$. Then $a\equiv 3\pmod{4}$ must hold, for otherwise $[\sigma,\varphi_a]$ has order $2^e$ by \eqref{eqn:order}. Let $[\sigma^v,\varphi_b]\in G$ be arbitrary. For $v$ odd, we have $b\equiv 3\pmod{4}$ for the same reason, and so
\[ b-1\equiv 2\equiv 2v \hspace{-2mm}\pmod{4}.\]
For $v$ even, we have $ab^{-1}\equiv 3\pmod{4}$ again for the same reason because
\[[\sigma,\varphi_a][\sigma^v,\varphi_b]^{-1} = [\sigma^{1-vab^{-1}},\varphi_{ab^{-1}}]\in G.\]
This means that $b\equiv 1\pmod{4}$, and so
\[ b-1\equiv 0\equiv 2v \pmod{4}.\]
We have therefore shown the congruence \eqref{eqn:ua cong}, and by taking $v=0$, we deduce that $\mathrm{Stab}_G(1_N)$ is contained in $\langle\varphi_5\rangle$.
\end{proof}

Unlike the odd prime power case, a transitive subgroup of $\Hol(N)$ need not have an element of order $2^{e}$, and similarly a regular subgroup of $\Hol(N)$ need not be cyclic. This is why the even prime power case is much more difficult. 

\begin{lemma}\label{lem:Rump} A group of order $2^e$ is isomorphic to a regular subgroup of $\Hol(N)$ if and only if it contains a cyclic subgroup of index $2$.
\end{lemma}
\begin{proof} Since regular subgroups of $\Hol(N)$ correspond to group operations $\circ$ for which $(N,\cdot,\circ)$ is a brace (see \cite[Theorem 4.2]{GV}), where $\cdot$ is the group operation on $N$, this lemma is a restatement of part of \cite[Theorem 3]{Rump1}.
\end{proof}

Let $G$ be a subgroup of $\Hol(N)$ of order $2^{e+s}$. Then
\[\ |G\cap N| = \frac{|G||N|}{|\Hol(N)|}[\Hol(N):GN] = 2^{s+1}[\Hol(N):GN],\]
which in particular implies that
\begin{equation}\label{eqn:sigma G}
|G\cap N| \geq 2^{s+1},\mbox{ or equivalently }\sigma^{2^{e-s-1}} \in G.
\end{equation}
This simple observation will be useful in several arguments.

\begin{proposition}\label{prop:p=2}
    Let $G$ be a transitive subgroup of $\Hol(N)$ of order $2^{e+s}$.
    \begin{enumerate}[$(a)$]
    \item If $s=1$ and $G$ has an element of order $2^e$, then for every subgroup $H$ of $G$ of index $2^e$ with $C =\mathrm{Core}_G(H)$, the quotient group $G/C$ is isomorphic to a transitive subgroup $T$ of $\Hol(N)$ under an isomorphism that sends $H/C$ to the stabiliser $\mathrm{Stab}_T(1_N)$.
    \item If $s=1$ and $G$ has no element of order $2^e$, or if $s\geq 2$, then there exists a normal subgroup $H$ of $G$ of index $2^e$ such that $G/H$ is not even isomorphic to any transitive subgroup of $\Hol(N)$.
    \end{enumerate}
The case $s=0$ is irrelevant because then a subgroup of $G$ of index $2^e$ is trivial.
\end{proposition}
\begin{proof}[Proof of $(a)$] Let $[\sigma^u,\varphi_a]\in G$ have order $2^e$, where $u$ is odd and $a\equiv 1\pmod{4}$ by \eqref{eqn:order}. Let $R = \langle[\sigma^u,\varphi_a]\rangle$, which is normal in $G$ because it has index $2$. For any natural number $k$, observe that
\[ [\sigma^u,\varphi_a]^k\in \Aut(N) \iff \sigma^{uS(a,k)}=1 \iff k\equiv 0\hspace{-2mm}\pmod{2^e}\]
by \eqref{eqn:power} and \eqref{eqn:valuation2}, which implies that $\mathrm{Stab}_R(1_N)=1$. Since $R$ has order $2^e$, it follows that $R$ is regular. Letting $G'=\mathrm{Stab}_G(1_N)$, we also see that $G'\cap R=1$ and so $G = R\rtimes G'$ by order consideration. Now, let $H$ be any subgroup of $G$ of
index $2^e$, namely of order $2$, with $C = \mathrm{Core}_G(H)$. 

\medskip 

Suppose first that $H\cap R=1$, in which case $G = R\rtimes H$.
\begin{enumerate}[(1)]
\item If $H$ is normal, then $C = H$ and $G = R\times H$, so projection onto $R$ induces an isomorphism $G/C\simeq R$ that sends $H/C$ to $\mathrm{Stab}_R(1_N)=1$.
\item If $H$ is not normal, then $C=1$ and let $H =\langle [\sigma^w,\varphi_c]\rangle$. Since $R$ is regular, we can find $[\sigma^w,\varphi_d]\in R$, where $c\not\equiv d\pmod{2^e}$ because $H\cap R=1$. Then
\[ [\sigma^w,\varphi_d]^{-1}[\sigma^w,\varphi_{c}] = \varphi_{d^{-1}c} \in G'\]
is non-trivial. The conjugation actions of  $[\sigma^w,\varphi_c]$ and $\varphi_{d^{-1}c}$ have the same effect on the cyclic subgroup $R$ because their quotient lies in $R$. Thus
\[ \Phi: G\rightarrow G;\quad \Phi|_R=\mathrm{id}_R,\,\ \Phi([\sigma^w,\varphi_c]) = \varphi_{d^{-1}c}\]
defines an isomorphism, and it clearly sends $H$ to $G'$.
\end{enumerate}
This concludes the proof of the case $H\cap R =1$.

\medskip

Suppose now that $H\cap R = H$, in which case $H = \langle[\sigma^u,\varphi_a]^{2^{e-1}}\rangle = \langle \sigma^{2^{e-1}}\rangle$. Then $H$ is normal in $G$, that is $C=H$, and we have 
\[ G/C \simeq R/C\rtimes G'.\]
This is a non-cyclic group of order $2^e$ that contains the cyclic subgroup $R/C$ of index $2$. Lemma \ref{lem:Rump} yields that $G/C$ is isomorphic to a regular subgroup $T$ of $\Hol(N)$, and clearly $H/C$ is mapped to $\mathrm{Stab}_T(1_N)=1$ under any isomorphism.

\medskip

This completes the proof of (a).
\end{proof}

\begin{proof}[Proof of $(b)$] Since $\sigma^{2^{e-s-1}}\in G$ by \eqref{eqn:sigma G}, we may take $H = \langle \sigma^{2^{e-s}}\rangle$, which is a normal subgroup of $G$ of index $2^e$. We have $|H\cap N| = |H| = 2^s$, so under the hypothesis of (b), we see from Lemma \ref{lem:intersection} that $G/H$ has no element of order $2^{e-1}$. Thus, it follows from Lemma \ref{lem:transitive2} that $G/H$ is not even isomorphic to any transitive subgroup of $\Hol(N)$.
\end{proof}

Let $G$ be a transitive subgroup of $\Hol(N)$ of order $2^{e+s}$. By Proposition \ref{prop:p=2}, we know that the answer to Question \ref{Q:HSG'} is ``yes" for every $H$ if  $s=1$ and $G$ has an element of order $2^e$, and ``no" for some $H$ otherwise. We shall now give a complete characterisation of such $H$ in the latter case. We have three possible situations, depending on whether
\[ |H\cap N| \geq 4,\quad |H\cap N| =2,\quad |H\cap N|=1,\]
and they require different arguments. The case $|H\cap N| \geq 4$ is easy.
 
\begin{proposition}\label{prop:4}
    Let $G$ be any subgroup of $\Hol(N)$ and let $H$ be any subgroup of $G$ of index $2^e$ with $C=\mathrm{Core}_G(H)$. For $|H\cap N|\geq 4$, the group $G/C$ is not isomorphic to any transitive subgroup of $\Hol(N)$.
\end{proposition}
\begin{proof}
    This follows immediately from Lemmas \ref{lem:intersection}(b) and \ref{lem:transitive2}.
\end{proof}

Next, we deal with the case $|H\cap N|=2$. Our idea is to consider the centre and the commutator subgroup. For any $[\sigma^u,\varphi_a],[\sigma^v,\varphi_b]\in \Hol(N)$, we have
\begin{equation}\label{eqn:commutator}
[\sigma^v,\varphi_b][\sigma^u,\varphi_a][\sigma^v,\varphi_b]^{-1}[\sigma^u,\varphi_a]^{-1} = \sigma^{u(b-1)-v(a-1)}.
\end{equation}
This implies that $[\sigma^u,\varphi_a]$ and $[\sigma^v,\varphi_b]$ commute if and only if
\begin{equation}\label{eqn:commute}
 u(b-1) \equiv v(a-1)\hspace{-2mm}\pmod{2^e}.
\end{equation}
In particular, we see that
\begin{equation}\label{eqn:centre}
\sigma^{2^{e-1}} \in Z(\Hol(N))\, \mbox{ and }\, \sigma^{2^{e-2}}\in Z(N\rtimes \langle\varphi_5\rangle). 
\end{equation}
Using these observations, we prove two important lemmas.

\begin{lemma}\label{lem:center}
    Let $G$ be a non-regular transitive subgroup of $\Hol(N)$.
    \begin{enumerate}[$(a)$]
    \item $Z(G)$ contains the element $\sigma^{2^{e-1}}$ of order $2$ and is cyclic.
    \item $Z(G)$ contains the element $[\sigma^{2^{e-2}},\varphi_{1+2^{e-1}}]$ of order $4$ when $G$ has no element of order $2^e$.
    \end{enumerate}
\end{lemma}
\begin{proof} Note that $\sigma^{2^{e-1}}\in Z(G)$ always holds by \eqref{eqn:sigma G} and \eqref{eqn:centre}. Also 
$\sigma^{2^{e-2}}\in G$ again by \eqref{eqn:sigma G} because $|G| \geq 2^{e+1}$ by non-regularity.

\medskip

To prove (a), it suffices to show that $Z(G)$ has a unique element of order $2$. Suppose that $[\sigma^u,\varphi_a]\in Z(G)$ is an element of order $2$ other than $\sigma^{2^{e-1}}$, which means that we have the congruences
\[
u(1+a)\equiv 0\hspace{-2mm}\pmod{2^e},\quad a^2\equiv1\hspace{-2mm}\pmod{2^e},\quad a\not\equiv 1\hspace{-2mm}\pmod{2^e}.\]
Since $G$ is transitive, we can find $[\sigma,\varphi_b]\in G$, and \eqref{eqn:commute} implies that
\[ u(b-1)\equiv a-1\hspace{-2mm}\pmod{2^e}.\]
If $a\equiv 3\pmod{4}$, then $u$ must be odd.
But for any $\varphi_c\in \mathrm{Stab}_G(1_N)$, we again see from \eqref{eqn:commute} that
\[ u(c-1)\equiv 0\hspace{-2mm}\pmod{2^e},\mbox{ that is }c\equiv1\hspace{-2mm}\pmod{2^e},\]
which contradicts that $G$ is non-regular. If $a\equiv 1\pmod{4}$, then $a\equiv 1+2^{e-1}\pmod{2^e}$ with $e\geq 3$ is the only possibility. But then
\[ 2u(1+2^{e-2})\equiv 0\hspace{-2mm}\pmod{2^e}\, \mbox{ and } \, u(b-1)\equiv 2^{e-1}\hspace{-2mm}\pmod{2^e},\]
which cannot simultaneously hold. 

\medskip

To prove (b), suppose that $G$ has no element of order $2^e$. Then $\mathrm{Stab}_G(1_N)$ is contained in $\langle\varphi_5\rangle$ by Lemma \ref{lem:transitive2}. Since $\mathrm{Stab}_G(1_N)\neq 1$ by non-regularity, we see that $e\geq 3$ necessarily and $\varphi_{1+2^{e-1}}\in G$, so in particular
\[[\sigma^{2^{e-2}},\varphi_{1+2^{e-1}}]\in G,\]
which is an element of order $4$ by \eqref{eqn:order}. We have
\[ 2^{e-2}(b-1)\equiv 2^{e-2}(2v)\equiv v((1+2^{e-1})-1)\hspace{-2mm}\pmod{2^e}\]
for all $[\sigma^v,\varphi_b]\in G$ by \eqref{eqn:ua cong}, whence $[\sigma^{2^{e-2}},\varphi_{1+2^{e-1}}]\in Z(G)$ by \eqref{eqn:commute}.
\end{proof}

\begin{lemma}\label{lem:center-commutator}
 Let $G$ be a non-regular transitive subgroup of $\Hol(N)$. Then 
 \[|Z(G)|\cdot |[G,G]| =  2^e.\]
\end{lemma}
\begin{proof} First, we prove the inequality
\[ |Z(G)|\cdot |[G,G]| \leq 2^e.\]
Put $|Z(G)|=2^r$, and note that it suffices to show that $[G,G]$ lies in $\langle \sigma^{2^{r}}\rangle$. By
Lemma \ref{lem:center}, we know that $Z(G)$ is cyclic, so let $[\sigma^u,\varphi_a]$ be its generator. We have $[\sigma^u,\varphi_a]^{2^{r-1}} = \sigma^{2^{e-1}}$ because $\sigma^{2^{e-1}}$ is the only element of order $2$ in $Z(G)$. By \eqref{eqn:power} and \eqref{eqn:valuation2}, this implies that
\[uS(a,2^{r-1}) \equiv 2^{e-1}\hspace{-2mm}\pmod{2^e}\, \mbox{ and }\, a^{2^{r-1}}\equiv 1\hspace{-2mm}\pmod{2^e}.\]
Let us define the integer constants
\[ x = \frac{uS(a,2^{r-1})}{2^{e-1}}\, \mbox{ and }\, y =\frac{a^{2^{r-1}}-1}{2^{e-1}},\]
where $x$ is odd and $y$ is even by the two congruences above. Since $2^{r-1}$ divides $S(a,2^{r-1})$ by \eqref{eqn:valuation2}, for any $[\sigma^v,\varphi_b]\in G$,  multiplying  \eqref{eqn:commute} by $S(a,2^{r-1})$ yields
\[ uS(a,2^{r-1})(b-1) \equiv v(a^{2^{r-1}}-1)\hspace{-2mm}\pmod{2^{e+r-1}}.\]
Dividing this by $2^{e-1}$ and rearranging, we then obtain
\[ b-1 \equiv (x^{-1}y)v \hspace{-2mm}\pmod{2^{r}}.\]
For any $[\sigma^w,\varphi_d],[\sigma^v,\varphi_b],\in G$, the above congruence implies that
\[ w(b-1) - v(d-1) \equiv w(x^{-1}y)v - v(x^{-1}y)w \equiv 0\hspace{-2mm}\pmod{2^r}.\]
It now follows from \eqref{eqn:commutator} that $[G,G]$ lies inside $\langle \sigma^{2^{r}}\rangle$, as desired. 

\medskip

Next, we prove the inequality
\[ |Z(G)|\cdot |[G,G]| \geq 2^e.\]
Put $|[G,G]| = 2^t$, and note that it suffices to show that $Z(G)$ has an element of order $2^{e-t}$. For any $[\sigma^w,\varphi_d],[\sigma^v,\varphi_b]\in G$, we have
\begin{equation}\label{eqn:commutator mod} w(b-1)\equiv v(d-1)\hspace{-2mm}\pmod{2^{e-t}}\end{equation}
by \eqref{eqn:commutator} because $[G,G] =\langle \sigma^{2^{e-t}}\rangle$ here. We consider two cases.
\begin{enumerate}[(1)]
\item Suppose that $G$ has an element $[\sigma^w,\varphi_d]$ of order $2^e$. Then 
\[ [\sigma^w,\varphi_d]^{2^t} = [\sigma^{wS(d,2^t)},\varphi_{d^{2^t}}]\]
has order $2^{e-t}$. Since $2^t$ divides $S(d,2^t)$ by \eqref{eqn:valuation2}, for any $[\sigma^v,\varphi_b]\in G$, by multiplying the congruence \eqref{eqn:commutator mod} by $S(d,2^t)$, we see that
\[ wS(d,2^t)(b-1) \equiv v (d^{2^t}-1)\hspace{-2mm}\pmod{2^e}.\]
It then follows from \eqref{eqn:commute} that $[\sigma^{wS(d,2^t)},\varphi_{d^{2^t}}]\in Z(G)$. 
\item Suppose that $G$ has no element of order $2^e$. Since $G$ is transitive, we can find $[\sigma,\varphi_c],[\sigma^{-1},\varphi_f]\in G$, and $c,f\equiv 3\pmod{4}$ by \eqref{eqn:ua cong}. Note that
\[ -(c-1) \equiv f-1\hspace{-2mm}\pmod{2^{e-t}}  \]
by \eqref{eqn:commutator mod}, so in particular
\begin{align*} (c-1)(f-1) &\equiv cf - 1\hspace{-2mm}\pmod{2^{e-t}},\\
(c-1)(f-1) & \equiv cf -1 \mbox{ or }cf-1+ 2^{e-t}\hspace{-2mm}\pmod{2^{e-t+1}}.
\end{align*}
Let us choose $\epsilon\in\{1,1+2^{e-1}\}$ to be such that 
\[ (c-1)(f-1)  \equiv (cf-1) + \frac{\epsilon-1}{2^{t-1}}\hspace{-2mm}\pmod{2^{e-t+1}}.\]
As in Lemma \ref{lem:center}(b), we have $\varphi_{1+2^{e-1}}\in G$ with $e\geq 3$ because $\mathrm{Stab}_G(1_N)$ lies in $\langle\varphi_5\rangle$ by Lemma \ref{lem:transitive2} and is non-trivial by non-regularity. Thus
\[ \left([\sigma,\varphi_c][\sigma^{-1},\varphi_f] \right)^{2^{t-1}}\varphi_{\epsilon} = [\sigma^{(1-c)S(cf,2^{t-1})},\varphi_{(cf)^{2^{t-1}}\epsilon}] \in G.\]
Note that $t\leq e-1$, for otherwise $[G,G]=N$ by \eqref{eqn:commutator} and $G$ would have an element of order $2^e$. Since $2^{t-1}$ exactly divides $S(cf,2^{t-1})$ by \eqref{eqn:valuation2} and $\varphi_{\epsilon}^2=1$, it is easy to see from \eqref{eqn:order} that this element has order $2^{e-t}$.

\medskip

Now, for any $[\sigma^v,\varphi_b]\in G$, we know from \eqref{eqn:commutator mod} that
\[ -(b-1) \equiv v(f-1)\hspace{-2mm}\pmod{2^{e-t}}.\]
Multiplying this congruence by $c-1$ then yields
\[ (1-c)(b-1) \equiv v\left( (cf-1) + \frac{\epsilon-1}{2^{t-1}}\right)\pmod{2^{e-t+1}}.\]
Since $2^{t-1}$ exactly divides $S(cf,2^{t-1})$ by \eqref{eqn:valuation2}, we then obtain
\begin{align*}
&\hspace{5mm}(1-c)S(cf,2^{t-1})(b-1)\\ & \equiv v\left( ((cf)^{2^{t-1}}-1)
    + \frac{S(cf,2^{t-1})}{2^{t-1}}(\epsilon-1)\right)&&\hspace{-5mm}\pmod{2^{e}}\\
    &\equiv v\left(  ((cf)^{2^{t-1}}\epsilon -1) + \left( \frac{S(cf,2^{t-1})}{2^{t-1}} - (cf)^{2^{t-1}}\right)(\epsilon-1) \right)&&\hspace{-5mm}\pmod{2^e}\\
    &\equiv v((cf)^{2^{t-1}}\epsilon-1)&&\hspace{-5mm}\pmod{2^e},
\end{align*}
where the last congruence holds because $\epsilon\in \{1,1+2^{e-1}\}$ and
\[ \frac{S(cf,2^{t-1})}{2^{t-1}}\equiv 1 \equiv (cf)^{2^{t-1}} \hspace{-2mm}\pmod{2}.\]
We now deduce from \eqref{eqn:commute} that $[\sigma^{(1-c)S(cf,2^{t-1})},\varphi_{(cf)^{2^{t-1}}\epsilon}]\in Z(G)$.
\end{enumerate}
In both cases, we exhibited an element of order $2^{e-t}$ in $Z(G)$, as desired.

\medskip

We have thus proven the desired equality.
\end{proof}

 
\begin{proposition}\label{prop:2} Let $G$ be a transitive subgroup of $\Hol(N)$ and let $H$ be any subgroup of $G$ of index $2^e$ with $C = \mathrm{Core}_G(H)$. For $|H\cap N| = 2$, the following are equivalent:
\begin{enumerate}[$(1)$]
\item $G/C$ is isomorphic to a transitive subgroup $T$ of $\Hol(N)$ under an isomorphism that sends $H/C$ to the stabiliser $\mathrm{Stab}_T(1_N)$.
\item $G$ has an element of order $2^e$ and $H$ is normal in $G$.
\end{enumerate}
\end{proposition}
\begin{proof} Note that $|H\cap N|=2$ means $H\cap N = \langle\sigma^{2^{e-1}}\rangle$.

\medskip 

First, suppose that $H$ is normal in $G$, that is $C = H$. Then $|G/C|=2^e$, so (1) states that $G/C$ is isomorphic to a regular subgroup of $\Hol(N)$. 
Thus, it follows from Lemma \ref{lem:Rump} that (1) occurs exactly when $G/C$ has an element of order $2^{e-1}$, which in turn is equivalent to $G$ having an element of order $2^e$ by Lemma \ref{lem:intersection}(c).

\medskip

Now, suppose that $H$ is not normal in $G$, that is $C\subsetneq H$. Let us assume for contradiction that $G/C$ is isomorphic to a transitive subgroup, which must be non-regular by order consideration, of $\Hol(N)$. Note that $G/C$ has no element of order $2^e$ by Lemma \ref{lem:intersection}(a), so necessarily $Z(G/C)$ is cyclic of order at least $4$ by Lemma \ref{lem:center}. Below, we shall show that
\begin{enumerate}[(i)]
\item $[Z(G/C) : Z(G)C/C]\leq 2$
\item $[Z(G/C):Z(G)C/C]\geq 4$
\end{enumerate}
simultaneously hold, which would lead to a contradiction.

\medskip

To prove (i), since $Z(G/C)$ is cyclic, it suffices to show that 
\[[\sigma^u,\varphi_a]^2=[\sigma^{u(1+a)},\varphi_{a^2}]\in  Z(G)\]
for all $[\sigma^u,\varphi_a]C\in Z(G/C)$. Indeed, for any $[\sigma^v,\varphi_b]\in G$, by \eqref{eqn:commutator} we have
\[ \sigma^{u(b-1)-v(a-1)}\in C,\, \mbox{ that is } \, u(b-1) \equiv v(a-1)\hspace{-2mm}\pmod{2^{e-1}}\]
because $H\cap N = \langle\sigma^{2^{e-1}}\rangle$. But then
\[ u(1+a)(b-1) \equiv v(a^2-1)\hspace{-2mm}\pmod{2^{e}}\]
and so $[\sigma^{u(1+a)},\varphi_{a^2}]\in Z(G)$
by \eqref{eqn:commute}, as desired. 

\medskip

To prove (ii), 
recall that the transitive subgroup of $\Hol(N)$ to which $G/C$ is assumed to be isomorphic is non-regular by order consideration, and $G$
is also non-regular similarly. We may then apply Lemma \ref{lem:center-commutator} to obtain
\[  |Z(G)|\cdot |[G,G]| = 2^e = |Z(G/C)|\cdot |[G/C,G/C]|.\]
Noting that $[G/C,G/C] = [G,G]C/C$, we can use the above equality to rewrite
\begin{align*} [Z(G/C) : Z(G)C/C] &= \frac{|[G,G]|}{|[G,G]C/C|}\cdot \frac{|Z(G)|}{|Z(G)C/C|}\\
&=|[G,G]\cap C| \cdot |Z(G)\cap C|.\end{align*}
Note that $\sigma^{2^{e-1}}\in Z(G)$ by \eqref{eqn:centre}, and $\sigma^{2^{e-1}}\in [G,G]$ because $[G,G]$ is a subgroup of $N$ by \eqref{eqn:commutator} and is non-trivial by the non-normality of $H$. But clearly $H\cap N\subseteq C$ because the subgroups of $N$ are all characteristic. Hence, both of the factors above are at least $2$, and the index in question is at least $4$.

\medskip

We have thus shown both (i) and (ii), which is a contradiction. This means that $G/C$ cannot be isomorphic to any transitive subgroup of $\Hol(N)$.
\end{proof}

Finally, we deal with the case $|H\cap N|=1$.

\begin{lemma}\label{lem:cap1} Let $G$ be a transitive subgroup of $\Hol(N)$ and let $H$ be any subgroup of $G$ of index $2^e$ with $e \geq 3$. For $|H\cap N|=1$, the following hold:
\begin{enumerate}[$(a)$]
\item If $H$ is cyclic and different from the subgroup
$\langle [\sigma^u,\varphi_{-1}]\rangle$ of order $2$ for any odd integer $u$, then $H$ is conjugate to $\mathrm{Stab}_G(1_N)$ in $G$.
\item If $H$ is non-cyclic, then either $H$ is conjugate to $\mathrm{Stab}_G(1_N)$ in $G$, or $H$ can be mapped to $\mathrm{Stab}_G(1_N)$ under an outer automorphism of $G$.
\end{enumerate}
In particular, the core of $H$ in $G$ is trivial under the above hypotheses.
\end{lemma}

In what follows, let $|H| =2^s$, and we can assume that $s\geq 1$. Note that the projection of $H$ onto $\Aut(N)$ is isomorphic to $H$ because $H\cap N=1$. Hence, if $H$ is cyclic, then the projection is equal to
\[ \langle\varphi_a\rangle,\mbox{ where }\begin{cases}
   a\equiv 1+2^{e-s},\, -1+2^{e-s}\hspace{-2mm}\pmod{2^e} &\mbox{when }s\geq 2,\\
   a\equiv 1+2^{e-1},\, -1+2^{e-1},\, -1\hspace{-2mm}\pmod{2^e} &\mbox{when }s=1.
\end{cases}
\]
Note that $s\leq e-2$, namely $2^{e-s}\equiv 0\pmod{4}$, has to hold here, for otherwise the projection would be $\Aut(N)$, which is non-cyclic since $e\geq 3$. On the other hand, if $H$ is non-cyclic, then the projection is equal to
\[\langle \varphi_{-1}\rangle\times \langle \varphi_{a}\rangle,\mbox{ where }a\equiv 1+2^{e-s+1}\hspace{-2mm}\pmod{2^e},\]
because a non-cyclic subgroup of $\Aut(N)$ must contain $\langle\varphi_{-1}\rangle$.
Here $s\leq e-1$, namely $2^{e-s+1}\equiv 0\pmod{4}$, has to hold because $\Aut(N)$ has order $2^{e-1}$. 

\medskip 

Therefore, by lifting the generators to $H$, we can write
\[ H = \begin{cases}\langle[\sigma^u,\varphi_a]\rangle &\mbox{in (a)},\\
\langle [\sigma^w,\varphi_{-1}]\rangle\times \langle [\sigma^u,\varphi_{a}]\rangle&\mbox{in (b)}.\end{cases}\]
Also put $G' = \mathrm{Stab}_G(1_N)$ for brevity. We now proceed to the proof.

\begin{proof}[Proof of $(a)$] We use the same idea as in  the proof of Proposition \ref{prop:odd}. Since $G$ is transitive, for any $v$ we can find $[\sigma^v,\varphi_b]\in G$, and observe that
\[ [\sigma^v,\varphi_b][\sigma^u,\varphi_a][\sigma^{v},\varphi_b]^{-1} 
= [\sigma^{v(1-a) +ub},\varphi_a].
\]
Below, we show that $v$ may be taken to be such that
\[v_2(v(1-a)+ub) > v_2(u),\]
in which case we can repeat this process to deduce that $[\sigma^u,\varphi_a]$ is conjugate to $\varphi_a$. Since $|H| = |G'|$, we must then have $G'= \langle\varphi_a\rangle$.
\begin{enumerate}[(1)]
\item 
If $a\equiv 1+2^{e-s}\pmod{2^e}$, then we see from \eqref{eqn:power} that  
\[ [\sigma^u,\varphi_a]^{2^s} = 1\, \mbox{ implies }\, uS(a,2^s) \equiv 0\hspace{-2mm}\pmod{2^e}.\]
Since $2^{e-s}\equiv0\pmod{4}$ here, we deduce from \eqref{eqn:valuation2} that $u\equiv0\pmod{2^{e-s}}$. Thus, we can pick $v$ to satisfy $v(1-a)\equiv -u\pmod{2^e}$, and we have
\[ v(1-a)+ub \equiv u(b-1)\equiv 0\hspace{-2mm}\pmod{2^{v_2(u)+1}}.\]
\item If $a\equiv -1+2^{e-s}\pmod{2^e}$ , then we again see from \eqref{eqn:power} that
\[ ([\sigma^u,\varphi_a]^2)^{2^{s-1}} =1\, \mbox{ implies }\, u(1+a)S(a^2,2^{s-1})\equiv 0\hspace{-2mm}\pmod{2^e}.\]
Since $2^{e-s}\equiv0\pmod{4}$ here, we deduce from \eqref{eqn:valuation2} that $u$ must be even. Thus, we can pick $v=2^{v_2(u)-1}$, and we have
\[ v(1-a)+ub\equiv 2^{v_2(u)} \left(1 - 2^{e-s-1} + \frac{ub}{2^{v_2(u)}} \right) \equiv 0\hspace{-2mm}\pmod{2^{v_2(u)+1}},\]
where $2^{e-s-1}$ is even because $s\leq e-2$.
%
\item If $s=1$ and $a\equiv -1\pmod{2^{e}}$, then $u$ is even by hypothesis. Thus, we can similarly pick $v=2^{v_2(u)-1}$, and we have
\[ v(1-a)+ub\equiv 2^{v_2(u)} \left(1+\frac{ub}{2^{v_2(u)}} \right) \equiv 0\hspace{-2mm}\pmod{2^{v_2(u)+1}}.\]
\end{enumerate}
In all cases, we have exhibited a suitable choice of $v$ that satisfies the desired inequality, and this completes the proof.
\end{proof}

\begin{proof}[Proof of $(b)$] 

Since $a\equiv 1\pmod{4}$, the same argument  as in (a) shows that we can conjugate $[\sigma^u,\varphi_a]$ to $\varphi_a$ in $G$. Thus, we may assume that 
\[H = \langle[\sigma^w,\varphi_{-1}]\rangle\times\langle\varphi_a\rangle\]
up to conjugation in $G$. Since $[\sigma^w,\varphi_{-1}]$ and $\varphi_a$ commute, we must have
\[ w(a-1)\equiv0 \hspace{-2mm}\pmod{2^e},\mbox{ that is } v_2(w) \geq s-1.\]
Note that $s\geq 2$ here because $H$ is non-cyclic.
In particular, $w$ is even, so as in (a), we can find $[\sigma^{2^{v_2(w)-1}},\varphi_b]\in G$ by transitivity, and 
\[ [\sigma^{2^{v_2(w)-1}},\varphi_b] [\sigma^w,\varphi_{-1}][\sigma^{2^{v_2(w)-1}},\varphi_b]^{-1} = [\sigma^{2^{v_2(w)} +wb},\varphi_{-1}],\]
where we have 
\[v_2(2^{v_2(w)} +wb) = v_2(w) + v_2\left(1+\frac{wb}{2^{v_2(w)}} \right) > v_2(w).\]
By repeating this process, we can then conjugate $[\sigma^w,\varphi_{-1}]$ to $\varphi_{-1}$ in $G$. However, we must also track how the element $\varphi_a$ gets affected in the process. Since $a\equiv 1+2^{e-s+1}\pmod{2^e}$, for any $f\geq s-1$ we see that
\begin{align*} [\sigma^{2^{f-1}},\varphi_b]\varphi_a[\sigma^{2^{f-1}},\varphi_b]^{-1} &= [\sigma^{2^{f-1}(1-a)},\varphi_{a}] \\&= \begin{cases}
    \varphi_a & \mbox{for }f\geq s,\\
    [\sigma^{2^{e-1}},\varphi_a]& \mbox{for }f=s-1.
\end{cases} \end{align*}
Therefore, we deduce that:
\begin{enumerate}[(1)]
\item If $v_2(w) \geq s$, then $\varphi_a$ is not affected in the process, and so $H$ is conjugate to $\langle\varphi_{-1}\rangle\times \langle\varphi_a\rangle$ in $G$. Since $|H| = |G'|$, we must have $G' = \langle\varphi_{-1}\rangle\times \langle\varphi_a\rangle$. 
\item If $v_2(w)=s-1$, then $\varphi_a$ is conjugated to $[\sigma^{2^{e-1}},\varphi_a]$ at the first step, but $[\sigma^{2^{e-1}},\varphi_a]$ remains unchanged afterwards by \eqref{eqn:centre} and the $f\geq s$ case.
\end{enumerate}
In case (1), we are done. In case (2), we may assume that
\[ H = \langle\varphi_{-1}\rangle\times \langle [\sigma^{2^{e-1}},\varphi_a]\rangle\]
up to conjugation in $G$. Since $\sigma^{2^{e-1}}\in G$ by \eqref{eqn:sigma G}, we deduce that $\varphi_a\in G$, and
\[G'=\langle\varphi_{-1}\rangle\times\langle\varphi_a\rangle\]
because $|H| = |G'|$. Below, we construct an automorphism of $G$ that sends $H$ to $G'$. Note that we can find $[\sigma,\varphi_c]\in G$ by transitivity, and we have
\[ [\sigma,\varphi_c]\varphi_{-1}[\sigma,\varphi_c]^{-1}\varphi_{-1}^{-1} = \sigma^2 \in G.\]
This implies that $[N:G\cap N]=1,2$. We consider these two cases separately.

\medskip

If $[N:G\cap N]=1$, then $N$ lies in $G$ and by order consideration, we obtain
\[ G = I\rtimes \langle [\sigma^{2^{e-1}},\varphi_{a}]\rangle = I\rtimes \langle\varphi_{a}\rangle,\mbox{ where } I= N\rtimes \langle\varphi_{-1}\rangle. \]
The conjugation actions of $[\sigma^{2^{e-1}},\varphi_a]$ and $\varphi_{a}$ plainly have the same effect on $I$ because $\sigma^{2^{e-1}}\in Z(G)$ by \eqref{eqn:centre}. It follows that
\[ \Phi: G\rightarrow G;\quad \Phi|_I = \mathrm{id}_I,\,\  \Phi([\sigma^{2^{e-1}},\varphi_{a}])=\varphi_a\]
defines an automorphism on $G$, and it clearly sends $H$ to $G'$.

\medskip

If $[N:G\cap N]=2$, then the projection of $G$ onto $\Aut(N)$ has order
\[ [G:G\cap N] = \frac{2^{e+s}}{2^{e-1}} = 2^{s+1} = 2|H| = 2[H:H\cap N],\]
so it contains the projection of $H$ onto $\Aut(N)$ as a subgroup of index $2$. The projection of $G$ onto $\Aut(N)$ must then be equal to
\[\langle \varphi_{-1}\rangle \times \langle \varphi_{\widetilde{a}}\rangle,\mbox{ where }\widetilde{a}\equiv 1+2^{e-s}\hspace{-2mm}\pmod{2^e}.\]
Let $z$ be such that $[\sigma^z,\varphi_{\widetilde{a}}]\in G$. Note that $z$ is odd, for otherwise $\sigma^{z}\in G$ and $\varphi_{\widetilde{a}}=\sigma^{-z}[\sigma^z,\varphi_{\widetilde{a}}]\in G'$, which is not the case. 
We then deduce that $[\sigma^z,\varphi_{\widetilde{a}}]$ has order $2^{e}$ by \eqref{eqn:order}, and that $N\cap \langle[\sigma^z,\varphi_{\widetilde{a}}]\rangle  =\langle\sigma^{2^{s}}\rangle$ by \eqref{eqn:power} and \eqref{eqn:valuation2}. We shall also choose $z$ to be such that $z\equiv 3\pmod{4}$, which is possible because $\sigma^2\in G$. This condition will be important for the later calculations.

\medskip

Now, let us consider the product
\[ J = (G\cap N)\langle [\sigma^z,\varphi_{\widetilde{a}}]\rangle = \langle\sigma^2,[\sigma^z,\varphi_{\widetilde{a}}]\rangle,\]
which is a subgroup of $G$ because $G\cap N$ is normal in $G$. We have
\[ |J| = \frac{|G\cap N||\langle[\sigma^z,\varphi_{\widetilde{a}}]\rangle|}{|N\cap \langle [\sigma^z,\varphi_{\widetilde{a}}]\rangle|} = \frac{2^{e-1}\cdot 2^e}{2^{e-s}} = 2^{e+s-1},\]
and so $G = J\rtimes \langle\varphi_{-1}\rangle$ has to hold. Consider 
\begin{align*} \Phi : G\rightarrow G;\quad &\Phi(\sigma^2)=\sigma^{2(1+2^{e-3})},\,\ \Phi(\varphi_{-1}) = \varphi_{-1}\varphi_{a}^{2^{s-2}} = \varphi_{-1}\varphi_{1+2^{e-1}},\\
&\Phi([\sigma^{z},\varphi_{\widetilde{a}}])=\sigma^{2^{e-3}}[\sigma^{z},\varphi_{\widetilde{a}}] = [\sigma^{2^{e-3}+z},\varphi_{\widetilde{a}}].\end{align*}
Note that $e\geq 4$ here, for otherwise $G$ would contain $N$ because $s\geq 2$. Hence, we have $\sigma^{2^{e-3}} =(\sigma^2)^{2^{e-4}}\in G$, and $\sigma^{2^{e-3}}[\sigma^{z},\varphi_{\widetilde{a}}]$ has order $2^e$ by \eqref{eqn:order}. 
Since $G\cap N$ is normal in $J$ and is centralised by $\sigma^{2^{e-3}}$, we easily check that $\Phi$ defines a homomorphism on $J$. Moreover, we have
\begin{align*}
\Phi(\varphi_{-1})\Phi(\sigma^2)\Phi(\varphi_{-1})^{-1} 
&=\sigma^{-2(1+2^{e-1})(1+2^{e-3})}\\
   &= \sigma^{-2(1+2^{e-3})}\\
   &= \Phi(\varphi_{-1}\sigma^2\varphi_{-1}^{-1}),
   \end{align*}
and observe that
\begin{align*}
\Phi(\varphi_{-1})\Phi([\sigma^z,\varphi_{\widetilde{a}}])\Phi(\varphi_{-1})^{-1} 
    &= \sigma^{-(1+2^{e-1})(2^{e-3}+z)}\cdot \varphi_{\widetilde{a}}\\
    &= \sigma^{-(2+2^{e-1})(2^{e-3}+z)}\cdot[\sigma^{2^{e-3}+z},\varphi_{\widetilde{a}}]\\
    & = \sigma^{-2z(1+2^{e-3})}\cdot\Phi([\sigma^z,\varphi_{\widetilde{a}}])\\
    & = \Phi(\sigma^{2})^{-z}\cdot \Phi([\sigma^z,\varphi_{\widetilde{a}}])\\
    & = \Phi(\varphi_{-1}[\sigma^z,\varphi_{\widetilde{a}}]\varphi_{-1}^{-1}),
\end{align*} 
where the third last equality holds since
$z\equiv 3\pmod{4}$. Hence, we have shown that $\Phi$ defines a homomorphism on $G$. It is not hard to see that $\mathrm{Im}(\Phi)$ contains all three of the generators $\sigma^2, [\sigma^z,\varphi_a],  \varphi_{-1}$ of $G$, so in fact $\Phi$ is an automorphism on $G$. Finally, note that $a\equiv \widetilde{a}^{2x}\pmod{2^{e}}$ for some odd $x$, so we see that
\begin{align*}
\Phi([\sigma^{2^{e-1}},\varphi_a]) & = \Phi(\sigma^{2^{e-1}-zS(\widetilde{a},2x)}\cdot [\sigma^z,\varphi_{\widetilde{a}}]^{2x})\\
& = \sigma^{(2^{e-1}-zS(\widetilde{a},2x))(1+2^{e-3})}\cdot (\sigma^{2^{e-3}}[\sigma^z,\varphi_{\widetilde{a}}])^{2x}\\
& = [\sigma^{(2^{e-1} - zS(\widetilde{a},2x))(1+2^{e-3})+(2^{e-3}+z)S(\widetilde{a},2x)},\varphi_{a}]\\
& = [\sigma^{2^{e-1} - 2^{e-3}S(\widetilde{a},2x)(z-1)},\varphi_a]\\
& = \varphi_a,
\end{align*}
where in the last equality, we used the choice that $z\equiv 3\pmod{4}$ and the fact that $v_2(S(\widetilde{a},2x))=1$ by \eqref{eqn:valuation2}. It follows that $\Phi$ takes $H$ to $G'$, as desired.

\medskip

This concludes the proof.
\end{proof}

To deal with the remaining case when $H = \langle[\sigma^u,\varphi_{-1}]\rangle$ with $u$ odd, we shall compare the centraliser of $H$ with that of the stabilisers.

\begin{lemma}\label{lem:centraliser} Let $G$ be a transitive subgroup of $\Hol(N)$ of order $2^{e+1}$ that contains $\varphi_{1+2^{e-1}}$. Then we have
\[|C_G(\varphi_{1+2^{e-1}})| = 2^e.\]
Moreover, for any $[\sigma^u,\varphi_{-1}]\in G$ with $u$ odd (if it exists), we have
\[|C_G([\sigma^u,\varphi_{-1}])| \leq 2^e,\]
which is a strict inequality if and only if 
\begin{equation}\label{eqn:non-cong} u(b-1) \not\equiv -2v \hspace{-2mm}\pmod{2^{e-1}}\mbox{ for some }[\sigma^v,\varphi_b]\in G.\end{equation}
\end{lemma}

\begin{proof}
The hypothesis implies that $\mathrm{Stab}_G(1_N) = \langle \varphi_{1+2^{e-1}}\rangle$, and for each $v$, we have exactly one $\varphi_b$ modulo $\langle\varphi_{1+2^{e-1}}\rangle$ such that $[\sigma^v,\varphi_b]\in G$. 

\medskip

For any $[\sigma^v,\varphi_b]\in G$, it follows from \eqref{eqn:commute} that 
\begin{align}\notag 
[\sigma^v,\varphi_b]\in C_G(\varphi_{1+2^{e-1}}) &\iff 0\equiv 2^{e-1}v\hspace{-2mm}\pmod{2^e},\\\label{eqn:centralizer}
[\sigma^v,\varphi_b]\in C_G([\sigma^u,\varphi_{-1}])& \iff u(b-1)\equiv -2v\hspace{-2mm}\pmod{2^e}.
\end{align}
For $\varphi_{1+2^{e-1}}$, the condition simply says that $v$ is even, so there are $2^{e-1}$ choices for $v$ and the equality follows. For $[\sigma^u,\varphi_{-1}]$, note that $\varphi_{1+2^{e-1}}$ does not satisfy the congruence \eqref{eqn:centralizer} because $u$ is odd, which implies that
\[ [\sigma^v,\varphi_{b}]\in C_G([\sigma^u,\varphi_{-1}])\, \mbox{ and }\, [\sigma^v,\varphi_{b(1+2^{e-1})}] \in C_G([\sigma^u,\varphi_{-1}])\]
cannot hold simultaneously. This observation yields the desired inequality, and it also implies that the inequality is strict if and only if there exists $[\sigma^v,\varphi_b]\in G$ such that both of the containments fail, namely
\[ [\sigma^v,\varphi_{b}],[\sigma^v,\varphi_{b(1+2^{e-1})}] \not\in C_G([\sigma^u,\varphi_{-1}]).\]
Since $u$ is odd, this is equivalent to $u(b-1)\not\equiv -2v\pmod{2^{e-1}}$ by \eqref{eqn:centralizer}.
\end{proof}

\begin{proposition}\label{prop:1}
    Let $G$ be a transitive subgroup of $\Hol(N)$ and let $H$ be any subgroup of $G$ of index $2^e$ with $C =\mathrm{Core}_G(H)$. For $|H\cap N|=1$, the following are equivalent:
    \begin{enumerate}[$(1)$]
    \item $G/C$ is not isomorphic to any transitive subgroup $T$ of $\Hol(N)$ under an isomorphism that sends $H/C$ to the stabiliser $\mathrm{Stab}_T(1_N)$.
    \item $|G| = 2^{e+1}$, $|G\cap N|\geq 8$ or $|G\cap N| = |[G,G]| =4$, $\mathrm{Stab}_G(1_N)=\langle\varphi_{1+2^{e-1}}\rangle$, and $H = \langle[\sigma^u,\varphi_{-1}]\rangle$ for an odd integer $u$.
    \end{enumerate}
\end{proposition}
\begin{proof}
We may assume that $e\geq 3$, because otherwise $G=\mathrm{Hol}(N)$ is the only non-regular transitive subgroup, in which case (1) fails  by Proposition \ref{prop:p=2}, and (2) also fails because $|G\cap N|=4,\, |[G,G]|=2$. Put $G'=\mathrm{Stab}_G(1_N)$.

\medskip

First, suppose that (1) holds. Then $|G| = 2^{e+1}$ and $H = \langle[\sigma^u,\varphi_{-1}]\rangle$ with $u$ odd by Lemma \ref{lem:cap1}. We also know from Proposition \ref{prop:p=2} that $G$ has no element of order $2^e$, and so $G' = \langle\varphi_{1+2^{e-1}}\rangle$ by Lemma \ref{lem:transitive2}. 

\medskip

Conversely, suppose that
\[ |G| = 2^{e+1},\quad G' = \langle\varphi_{1+2^{e-1}}\rangle,\quad H = \langle[\sigma^u,\varphi_{-1}]\rangle\mbox{ with $u$ odd}.\]
Note that $\varphi_{1+2^{e-1}}$ and $[\sigma^u,\varphi_{-1}]$ do not commute by \eqref{eqn:centralizer}, so $H$ is not normal in $G$, that is $C = 1$. We may then state the negation of (1) as follows:
\begin{enumerate}[($*$)]
\item $G$ is isomorphic to a transitive subgroup $T$ of $\Hol(N)$ under an isomorphism that sends $H$ to the stabiliser $\mathrm{Stab}_T(1_N)$.
\end{enumerate}
Note that $|G\cap N| \leq 2$ does not occur by \eqref{eqn:sigma G} and $|[G,G]|=1$ is also impossible by the non-normality of $H$. Since $[G,G]$ is contained in $G\cap N$ by \eqref{eqn:commutator}, there are only three cases:
\[ |G\cap N|\geq 8,\quad |G\cap N| = |[G,G]|=4,\quad |G\cap N| = 4\mbox{ with }|[G,G]|=2.\]
The claim of the proposition is that ($*$) fails in the first two cases, and holds in the last case. 

\medskip

Before considering each of the above cases, let us give a sufficient condition for ($*$) to fail. Note that $G$ has no element of order $2^e$. Indeed, if $[\sigma^v,\varphi_b]\in G$ is of order $2^e$, then $v$ is odd and $b\equiv 1\pmod{4}$ by \eqref{eqn:order}. But this yields $\sigma^2\in G$, because $u(\frac{b-1}{2})+v$ is odd and
\[ [\sigma^v,\varphi_b][\sigma^u,\varphi_{-1}][\sigma^v,\varphi_b]^{-1}[\sigma^u,\varphi_{-1}]^{-1} = \sigma^{u(b-1) + 2v} \in G\]
by \eqref{eqn:commutator}. Since $u$ and $v$ are both odd, this in turn implies that
\[ \varphi_{-b}=[\sigma^u,\varphi_{-1}]\cdot \sigma^{2(\frac{u-v}{2})} \cdot [\sigma^v,\varphi_{b}]\in G,\mbox{ where }-b\equiv 3\hspace{-2mm}\pmod{4},\]
and this contradicts the hypothesis on $G'$. Hence, if ($*$) holds, then by Lemma \ref{lem:transitive2} we must have $\mathrm{Stab}_T(1_N) =\langle\varphi_{1+2^{e-1}}\rangle$, and this implies that
\[|C_G([\sigma^u,\varphi_{-1}])| = |C_G(H)| = |C_T(\mathrm{Stab}_T(1_N))| = |C_T(\varphi_{1+2^{e-1}})|\]
must hold. By its contrapositive, we see that if the non-congruence \eqref{eqn:non-cong} holds, then the above equalities fail by Lemma \ref{lem:centraliser}, and so ($*$) also fails.


\medskip

For the case $|G\cap N|\geq 8$, that is $\sigma^{2^{e-3}}\in G$, the element $\sigma^{2^{e-3}}$ satisfies \eqref{eqn:non-cong} and so ($*$) does not hold.

\medskip

For the case $|G\cap N|=4$, that is $G \cap N = \langle\sigma^{2^{e-2}}\rangle$, observe that $G$ projects surjectively onto $\Aut(N)$, so we can find $[\sigma^z,\varphi_5]\in G$, and $z$ is even by \eqref{eqn:ua cong}.
\begin{enumerate}[$\bullet$]
\item For $e\geq 4$, we must have $v_2(z)=1$, for otherwise $\sigma^{zS(5,2^{e-4})} \in \langle\sigma^{2^{e-2}}\rangle \subseteq G$ by \eqref{eqn:valuation2}, which would imply that
\[\varphi_{5^{2^{e-4}}}= \sigma^{-zS(5,2^{e-4})}\cdot [\sigma^z,\varphi_{5}]^{2^{e-4}}  \in G.\]
This contradicts that $G'$ has order $2$.
\item For $e=3$, we have $\sigma^2 = \sigma^{2^{e-2}} \in G$, and so
\[ [\sigma^{z+2},\varphi_5] = \sigma^2[\sigma^z,\varphi_5]\in G. \]
Thus, replacing $z$ by $z+2$ if necessary, we may assume that $v_2(z)=1$.
\end{enumerate}
Since $v_2(z)=1$, we see from \eqref{eqn:order} that $[\sigma^z,\varphi_5]$ has order $2^{e-1}$, and from \eqref{eqn:power} and \eqref{eqn:valuation2} that $N\cap\langle[\sigma^z,\varphi_5]\rangle = \langle \sigma^{2^{e-1}}\rangle $.

\medskip

Now, similar to the proof of Lemma \ref{lem:cap1}(b), let us consider the product
\[ J = (G\cap N)\langle[\sigma^z,\varphi_5]\rangle = \langle \sigma^{2^{e-2}},[\sigma^z,\varphi_5]\rangle,\]
which is a subgroup of $G$ because $G\cap N$ is normal in $G$. We have
\[ |J| = \frac{|G\cap N||\langle[\sigma^z,\varphi_{5}]\rangle|}{|N\cap \langle [\sigma^z,\varphi_{5}]\rangle|} = \frac{2^2\cdot 2^{e-1}}{2} = 2^{e},\]
and so $G = J\rtimes H$ has to hold. Since  $J$ is abelian by \eqref{eqn:centre}, we see that
\begin{align*}
[G,G] = [J,J][J,H]\rtimes [H,H] = [J,H]. 
\end{align*}
Moreover, using \eqref{eqn:commutator}, we compute that
\begin{align*}
\sigma^{2^{e-2}}[\sigma^u,\varphi_{-1}]\sigma^{-2^{e-2}}[\sigma^u,\varphi_{-1}]^{-1} & =\sigma^{2^{e-1}} ,\\
[\sigma^z,\varphi_5][\sigma^u,\varphi_{-1}][\sigma^z,\varphi_{5}]^{-1}[\sigma^u,\varphi_{-1}]^{-1} &= \sigma^{4u+2z}.
\end{align*}
We then deduce (see \cite[Chapter 4, Exercise 2(a) in \S1]{Suzuki} for example) that
\[ [J,H] = \langle j\sigma^{2^{e-1}}j^{-1} , j\sigma^{4u+2z}j^{-1} : j \in J\rangle.\]
The conjugation action of $j\in J$ on $N$ clearly does not affect the subgroup that is being generated, so we see that
\[ |[G,G]| = |\langle \sigma^{2^{e-1}},\sigma^{4u+2z}\rangle| = \begin{cases}
4 & \mbox{when }4u+2z\not\equiv0\hspace{-2mm}\pmod{2^{e-1}},\\
2 &  \mbox{when }4u+2z\equiv0\hspace{-2mm}\pmod{2^{e-1}}.
\end{cases}\]
We consider these two cases separately.
\begin{enumerate}[(i)]
\item If $4u+2z\not\equiv 0 \pmod{2^{e-1}}$, then $[\sigma^z,\varphi_5]$ satisfies \eqref{eqn:non-cong} and so ($*$) does not hold, as we have already explained.
\item If $4u + 2z\equiv 0\pmod{2^{e-1}}$, then 
\[ 4u + 2z \equiv 0\hspace{-2mm}\pmod{2^e}\, \mbox{ or } \, 4u+2(z +2^{e-2}) \equiv 0\hspace{-2mm}\pmod{2^e},\]
so it follows from \eqref{eqn:centralizer} that 
\[ [\sigma^z,\varphi_{5}]\in C_G([\sigma^u,\varphi_{-1}])\, \mbox{ or }\, \sigma^{2^{e-2}}[\sigma^{z},\varphi_{5}] \in C_G([\sigma^u,\varphi_{-1}]).\]
For $e=3$, note that we have $4u+2z\equiv 0\pmod{8}$ since $u$ and $\frac{z}{2}$ are odd, so $[\sigma^z,\varphi_{5}]$ commutes with $[\sigma^u,\varphi_{-1}]$ by \eqref{eqn:centralizer}. For $e\geq 4$, since $2^{e-2}+ z$ is still exactly divisible by $2$, replacing $z$ by $2^{e-2}+z$ if needed, we may also assume that $[\sigma^z,\varphi_5]$ commutes with $[\sigma^u,\varphi_{-1}]$. Consider
\begin{align*} \Phi:G\rightarrow G;\quad &\Phi(\sigma^{2^{e-2}}) = [\sigma^u,\varphi_{-1}]\varphi_{1+2^{e-1}},\,\ \Phi([\sigma^u,\varphi_{-1}]) = \varphi_{1+2^{e-1}},\\&\Phi([\sigma^z,\varphi_5]) = [\sigma^z,\varphi_5].\end{align*}
Since $u$ is odd, we see from \eqref{eqn:order} that $[\sigma^u,\varphi_{-1}]\varphi_{1+2^{e-1}}$ has order $4$. Moreover,
we compute that
\begin{align*}
\Phi([\sigma^u,\varphi_{-1}])\Phi(\sigma^{2^{e-2}})\Phi([\sigma^u,\varphi_{-1}])^{-1} 
& = \varphi_{1+2^{e-1}}[\sigma^u,\varphi_{-1}]\\
& = ([\sigma^u,\varphi_{-1}]\varphi_{1+2^{e-1}})^{-1}\\
& = \Phi(\sigma^{2^{e-2}})^{-1}\\
&= \Phi([\sigma^u,\varphi_{-1}]\sigma^{2^{e-2}}[\sigma^u,\varphi_{-1}]^{-1}).
\end{align*}
Since $[\sigma^z,\varphi_5]$ commutes with both $\sigma^{2^{e-2}}$ and $\varphi_{1+2^{e-1}}$ by \eqref{eqn:commute}, and with $[\sigma^u,\varphi_{-1}]$ by our choice of $z$, the above is enough to conclude that $\Phi$ defines a homomorphism on $G$. Clearly $[\sigma^z,\varphi_5],[\sigma^u,\varphi_{-1}]\in \mathrm{Im}(\Phi)$, and we have
\[ \Phi([\sigma^z,\varphi_5]^{2^{e-3}} [\sigma^u,\varphi_{-1}]) = [\sigma^{zS(5,2^{e-3})},\varphi_{5^{2^{e-3}}}]\varphi_{1+2^{e-1}} = \sigma^{2^{e-2}x}\]
for some odd $x$ by \eqref{eqn:valuation2} because $v_2(z)=1$. This implies that $\Phi$ is in fact an automorphism of $G$, and it clearly sends $H$ to $G'$.
\end{enumerate}
The proof of the proposition is now complete.
\end{proof}

\section*{Acknowledgements}

We gratefully acknowledge the use of  \textsc{Magma} \cite{magma} in this research. Even though none of our proofs require computer calculations, some of the statements that we proved were discovered based on \textsc{Magma} computations.

\medskip

\noindent The first-named author has been supported by the following grants: 
\\\mbox{\hspace{1em}}Project OZR3762 of Vrije Universiteit
Brussel;
\\\mbox{\hspace{1em}}FWO Senior Research Project G004124N.
\\The second-named author recognises that a large part of this research was supported by her research funding provided by Ochanomizu University.

\Addresses

\end{document}